\documentclass[11pt,reqno,a4paper]{amsart}
\usepackage{hyperref}
\usepackage{latexsym}
\usepackage{amssymb,amsmath}
\usepackage{graphics}
\usepackage{url}
\usepackage{bbm}
\usepackage{multirow}
\usepackage[vcentermath]{youngtab}
  
\usepackage{booktabs}  
\usepackage{enumerate}
\usepackage[usenames,dvipsnames]{color}
\usepackage{dsfont}
\allowdisplaybreaks

\usepackage{braket} 
\usepackage{multirow}

\usepackage{tikz}
\usetikzlibrary{matrix,chains,shapes,arrows,scopes,positioning}
\usetikzlibrary{calc,decorations.pathmorphing,shapes}
\usetikzlibrary{patterns}
\tikzstyle{empty}=[circle,draw=black!80,thick]
\tikzstyle{emptyn}=[circle,draw=black!80,fill=white,scale=0.5] 
\tikzstyle{nero}=[circle,draw=black!80,fill=black!80,thick] 

\newcommand{\fS}{\mathfrak{S}}

\newcommand{\Irr}{{\mathrm{Irr}}}

\newcommand{\up}{\big\uparrow}

\newcommand{\cd}{\mathrm{cd}}
\newcommand{\lX}{\mathcal{X}}

\newtheorem{theorem}{Theorem}[section]
\newtheorem{lemma}[theorem]{Lemma}

\newtheorem{corollary}[theorem]{Corollary}
\newtheorem{proposition}[theorem]{Proposition}

\theoremstyle{definition}
\newtheorem{example}[theorem]{Example}

\newtheorem{definition}[theorem]{Definition}
\newtheorem{notation}[theorem]{Notation}
\newtheorem{remark}[theorem]{Remark}

\newcommand{\6}{^}

\begin{document}

\title[]{Non-linear Sylow Branching Coefficients for symmetric groups}


\author{Eugenio Giannelli}
\address[E. Giannelli]{Dipartimento di Matematica e Informatica U.~Dini, Viale Morgagni 67/a, Firenze, Italy}

\author{Giada Volpato}
\address[G. Volpato]{Dipartimento di Matematica e Informatica U.~Dini, Viale Morgagni 67/a, Firenze, Italy}
\email{eugenio.giannelli@unifi.it, giada.volpato@unifi.it}


\begin{abstract}
We study the restriction to Sylow subgroups of irreducible characters of symmetric groups. 
In particular, we focus our attention on constituents of degree greater than $1$. Our main result is a wide generalization of \cite[Theorem 3.1]{GN}.
\end{abstract}



\maketitle

\section{Introduction}\label{sec:intro}

The purpose of this article is to investigate the structure of the restriction to Sylow subgroups of irreducible characters of the symmetric group $\fS_n$.  
Let $p$ be a prime number and let $P_n$ be a fixed Sylow $p$-subgroup of $\fS_n$.
The main question studied in this paper is the following. 
Given $k\in\mathbb{N}$, which and how many irreducible characters of $\fS_n$ admit a constituent of degree $p\6k$ in their restriction to $P_n$? 
More formally, we let $\mathrm{Irr}_k(P_n)$ denote the set consisting of all the irreducible characters of $P_n$ of degree $p\6k$, and we focus our attention on the subset $\Omega_n\6k$ of $\mathrm{Irr}(\fS_n)$ defined as follows:
$$\Omega_n\6k=\{\chi\in\mathrm{Irr}(\fS_n)\ |\ [\chi_{P_n}, \phi]\neq 0,\ \text{for some}\ \phi\in\mathrm{Irr}_k(P_n)\}.$$

In \cite[Theorem 3.1]{GN} it is proved that the restriction to $P_n$ of any irreducible character of $\fS_n$ admits a linear constituent. In other words, $\Omega_n\60=\mathrm{Irr}(\fS_n)$.  This result was improved (for odd primes) in \cite{GL2} where, for every linear character $\phi$ of $P_n$, the authors classify those irreducible characters $\chi$ of $\fS_n$ such that $\phi$ appears as an irreducible constituent of $\chi_{P_n}$. 

In this article we largely extend in a new direction the result obtained in \cite{GN} mentioned above. 
More precisely, for any odd prime number $p$,  we are able to describe the set $\Omega_n\6k$, for any $k\in\mathbb{N}$. Surprisingly enough, these sets possess quite a regular structure. In order to describe it we recall that irreducible characters of $\fS_n$ are naturally in bijection with $\mathcal{P}(n)$, the set of partitions of $n$. With this in mind, we find it useful to think of $\Omega_n\6k$ as a subset of $\mathcal{P}(n)$ instead of $\mathrm{Irr}(\fS_n)$. For any $t\in\mathbb{N}$, we let $\mathcal{B}_n(t)$ be the subset of $\mathcal{P}(n)$ consisting of partitions whose Young diagram fits into a $t\times t$ grid (i.e. having first row and first column of size at most $t$). 

\medskip

The first main result of the article is \textbf{Theorem \ref{thm: boxes arb n}}, where we show that for any $k\in\mathbb{N}$ there exists a certain $T_n\6k\in\{1,\ldots, n\}$ such that $\Omega_n\6k=\mathcal{B}_n(T_n^k)$. In other words, Theorem  \ref{thm: boxes arb n} shows that the set of partitions of $n$ whose corresponding irreducible characters admit a constituent of degree $p^k$ on restriction to a Sylow $p$-subgroup of $\fS_n$ coincide with the set of partitions of $n$ which fit inside a square, whose size depend on both $n$ and $t$.
As mentioned above, this statement highlights the nice and well-behaved combinatorial structure of the sets $\Omega_n\6k$.  

The description given by Theorem \ref{thm: boxes arb n} is sharpened in \textbf{Theorem \ref{thm: n}}, where we explicitly compute the value of $T_n\6k$ for all $n,k\in\mathbb{N}$. We avoid the precise description of these values here, as it requires the introduction of some technical definitions. Nevertheless, we refer the reader to Tables \ref{table: p=3} and \ref{table: final} for several specific and concrete instances of our second main result. 

An intriguing consequence of Theorem \ref{thm: n} is that $\Omega_n^j\subseteq \Omega_n^k$, for all $k\leq j$. 
This means that whenever $\chi_{P_n}$ admits an irreducible constituent of degree $p^j$, then it also admits constituents of degree $p^k$, for all $0\leq k\leq j$. 

We conclude our article by studying the second part of the question we proposed above. Namely,  we give an estimate for how many characters of $\fS_n$ are contained in $\Omega_n\6k$. In \textbf{Corollary \ref{thm: asintotico}} we show that the restriction to $P_n$ of almost all irreducible characters of $\fS_n$ admits an irreducible constituent of degree $p\6k$, for all admissible $k\in\mathbb{N}$. More precisely, we prove that $$\lim_{n\rightarrow\infty}\frac{|\Omega_n|}{|\mathcal{P}(n)|}=1,$$
where $\Omega_n$ is the intersection of all of the sets $\Omega_n\6k$, where $k$ runs among all those natural numbers such that $p\6k$ is the degree of an irreducible character of $P_n$. 

\begin{remark}
As mentioned above, this article treats the case of odd primes. When $p=2$, linear constituents of the restriction to Sylow $2$-subgroups of odd degree characters of $\fS_n$ were studied in \cite{INOT}, mainly in connection with the McKay Conjecture \cite{Nav18}. 
Despite this, the object of our study seems to be particularly difficult when $p=2$. For instance, we immediately notice in this case that the set $\Omega_4\61=\{(3,1),(2,1,1)\}$ and therefore is not of the form $\mathcal{B}_4(T)$, for any $T\in \{1,2,3,4\}$. This shows that the main theorems of the present article do not hold for the prime $2$. 
Even if this irregularity might disappear for larger natural numbers, more serious obstacles arise in this setting. For example, Lemma \ref{lem: neq const} below asserts that the restriction to $P_n$ of every non-linear irreducible character of $\fS_n$ cannot admit a unique irreducible constituent of a certain degree. This is a crucial ingredient in the proofs of our main results. 
Unfortunately, this is plainly false when the prime is $2$. For instance, in \cite{GJLMS} it is shown that if $\lambda=(2\6n-x,1\6x)$ then $(\chi\6\lambda)_{P_{2\6n}}$ admits a unique constituent of degree $1$. Things can go even worse: if $\lambda=(2\6n-1,1)$ then it is not difficult to see that $(\chi\6\lambda)_{P_{2\6n}}$ admits a unique constituent of degree $2\6k$, for all $k\in\{0,1,\ldots, n-1\}$. 
\end{remark}

\noindent\textbf{Acknowledgments.}
We thank Stacey Law for carefully reading an earlier version of this article and for providing us with many helpful comments and suggestions. 
We are also grateful to the referees for several useful comments that improved our article.

\section{Notation and Background}\label{sec:prelim}
Throughout this article, $p$ denotes an odd prime. Given integers $n\leq m$, we denote by $[n,m]$ the set $\{n, n+1,\ldots, m\}$. If $n<m$ then $[m,n]$ is regarded as the empty set. 
We let $\mathcal{C}(n)$ be the set of compositions of $n$, i.e. the set consisting of all the finite sequences $(a_1,a_2,\ldots, a_z)$ such that $a_i$ is a non-negative integer for all $i\in [1,z]$ and such that $a_1+\cdots+a_z=n$.
Given $\lambda=(\lambda_1,\dots , \lambda_z)\in \mathcal{C}(n)$, we sometimes denote by $l(\lambda)=z$ the number of parts of $\lambda$.
As already mentioned in the introduction, $P_n$ denotes a Sylow $p$-subgroup of the symmetric group $\fS_n$. As usual, given a finite group $G$, we denote by $\mathrm{Irr}(G)$ the set of irreducible complex characters of $G$, and by $\mathrm{Lin}(G)$ the subset of linear characters of $G$. Finally, $\mathrm{cd}(G)=\{\chi(1)\ |\ \chi\in\mathrm{Irr}(G)\}$ is the set of irreducible character degrees. 

\subsection{Wreath products}
Here we fix the notation for characters of wreath products. For more details see \cite[Chapter 4]{JK}.
Let $G$ be a finite group and let $H$ be a subgroup of $\mathfrak{S}_n$. We denote by $G^{\times n}$ the direct product of $n$ copies of $G$. The natural action of $\mathfrak{S}_n$ on the direct factors of $G^{\times n}$ induces an action of $\mathfrak{S}_n$ (and therefore of $H\leq \mathfrak{S}_n $) via automorphisms of $G^{\times n}$, giving the wreath product $G\wr H := G^{\times n} \rtimes H$. We refer to $G^{\times n}$ as the base group of the wreath product $G\wr H$. 
We denote the elements of $G\wr H$ by $(g_1,\dots , g_n ;h)$ for $g_i\in G$ and $h\in H$. Let $V$ be a $\mathbb{C}G$-module and suppose it affords the character $\phi$.
We let $V^{\otimes n} := V \otimes \cdots \otimes V$ ($n$ copies) be the corresponding $\mathbb{C} G^{\times n}$-module. The left action of $G\wr H$ on $V^{\otimes n}$ defined by linearly extending $$(g_1,\dots ,g_n; h) : v_1 \otimes \cdots \otimes v_n \mapsto g_1 v_{h^{-1}(1)} \otimes \cdots \otimes g_n v_{h^{-1}(n) } ,$$
turns $V^{\otimes n}$ into a $\mathbb{C}(G\wr H)$-module, which we denote by $\tilde{V}^{\otimes n}$.
We denote by $\tilde{\phi}$ the character afforded by the $\mathbb{C}(G\wr H)$-module $\tilde{V}^{\otimes n}$.
For any character $\psi$ of $H$, we let $\psi$ also denote its inflation to $G\wr H$ and let 
$\mathcal{X} (\phi; \psi):= \tilde{\phi} \cdot \psi$
be the character of $G\wr H$ obtained as the product of $\tilde{\phi}$ and $\psi$.
Let $\phi \in {\rm Irr}(G)$ and let $\phi^{\times n}:= \phi \times \cdots \times \phi$ be the corresponding irreducible character of $G^{\times n}$. Observe that $\tilde{\phi}\in {\rm Irr}(G\wr H)$ is an extension of $\phi^{\times n}$.
Given $K\leq G$, we denote by ${\rm Irr}(G|\psi)$ the set of characters $\chi \in {\rm Irr}(G)$ such that $\psi$ is an irreducible constituent of the restriction $\chi_K$. Hence, by Gallagher's Theorem \cite[Corollary 6.17]{IBook} we have
\[{\rm Irr}(G\wr H | \phi^{\times n})= \Set{ \mathcal{X} (\phi; \psi) | \psi \in {\rm Irr}(H)}. \]

If $H=C_p$ is a cyclic group of prime order $p$, every $\psi \in \Irr(G\wr C_p)$ is either of the form
\begin{itemize}
	\item[(i)] $\psi= \phi_1 \times \cdots \times \phi_p \up^{G\wr C_p}_{G^{\times p}}$, where $\phi_1 , \dots \phi_p \in \Irr(G)$ are not all equal; or
	\item[(ii)] $\psi= \lX (\phi; \theta)$ for some $\phi \in \Irr(G)$ and $\theta \in \Irr(C_p)$. 
\end{itemize}
We remark that in case (i) we have that ${\rm Irr}(G\wr C_p | \phi_1 \times \cdots \times \phi_p)=\{\psi\}$. 

\subsection{Sylow subgroups of $\fS_n$}\label{sec2.2}
We record some facts about Sylow subgroups of symmetric group and we refer to \cite[Chapter 4]{JK} or to \cite{Olsson} for more details.

We let $P_n$ denote a Sylow $p$-subgroup of $\fS_n$. Clearly $P_1$ is the trivial group while $P_p\cong C_p$ is cyclic of order $p$. If $i\geq 2$, then $P_{p^i}=\big(P_{p^{i-1}}\big)^{\times p} \rtimes P_p=P_{p^{i-1}}\wr P_p\cong P_p\wr \cdots \wr P_p$ ($i$-fold wreath product). Let $n=\sum_{i=1}^t a_i p^{n_i}$ be the $p$-adic expansion of $n$. Then $P_n \cong P_{p^{n_1}}^{\times a_1} \times P_{p^{n_2}}^{\times a_2} \times \cdots \times P_{p^{n_t}}^{\times a_t}$.

For $n\in \mathbb{N}$, the normalizer of a Sylow $p$-subgroup of $\fS_{p^n}$ is $N_{\fS_{p^n}}(P_{p^n}) = P_{p^n} \rtimes H$, where $H\cong (C_{p-1})^{\times n}$.
More generally, if $n=\sum_{i=1}^t a_i p^{n_i}$, $n_t > \cdots > n_1 \geq 0$, then $N_{\fS_n}(P_n)=N_1 \wr \fS_{a_1} \times \cdots \times N_t \wr \fS_{a_t}$, where $N_i:=N_{\fS_{p^{n_i}}}(P_{p^{n_i}})$ for every $i\in [1,t]$.
We refer the reader to \cite[Section 2]{G21} for more details about the structure of the normaliser of a Sylow $p$-subgroup. The following fact is certainly well known. We state it here as we will need it in the following section of the article.


\begin{lemma} \label{lem: elements in N}
Let $p$ be an odd prime, 
let $n\in \mathbb{N}$ and let $H$ be a complement of $P_{p\6n}$ in $N_{\fS_{p^n}}(P_{p^n})$. 
There are no non trivial elements of $P_{p\6n}$ that are centralized by $H$.
\end{lemma}
\begin{proof}
This follows directly from the discussion in \cite[Section 2.2]{G21}.
\end{proof}

We remark that Lemma \ref{lem: elements in N} is equivalent to say that $C_{\fS_{p^n}}(P_{p^n})=Z(P_{p^n})$, for any $n\in\mathbb{N}$.

\subsection{The Littlewood-Richardson coefficients}
For each $n\in \mathbb{N}$, $\Irr(\fS_n)$ is naturally in bijection with $\mathcal{P}(n)$, the set of all partitions of $n$. For $\lambda \in \mathcal{P}(n)$, the corresponding irreducible character is denoted by $\chi^\lambda$.
Let $m,n\in \mathbb{N}$ with $m<n$. Given $\chi\6\mu\times\chi\6\nu\in\mathrm{Irr}(\fS_m\times\fS_{n-m})$, the decomposition into irreducible constituents of the induction
\[\left( \chi^\mu \times \chi^\nu \right)^{\mathfrak{S}_n}= \sum_{\lambda\in \mathcal{P}(n)} \mathcal{LR}(\lambda;\mu ,\nu) \chi^{\lambda}\]
is described by the Littlewood-Richardson rule (see \cite[Chapter 5]{Fulton} or \cite[Chapter 16]{J}).
Here the natural numbers $\mathcal{LR}(\lambda;\mu ,\nu)$ are called \emph{Littlewood-Richardson coefficients}.
Given $(n_1,\ldots, n_k)\in\mathcal{C}(n)$, $\lambda\in\mathcal{P}(n)$ and $\mu_j\in\mathcal{P}(n_j)$ for all $j\in [1,k]$, we let $\mathcal{LR}(\lambda; \mu_1,\ldots, \mu_k)$ be the multiplicity of $\chi^\lambda$ as an irreducible constituent of $(\chi^{\mu_1}\times\cdots\times\chi^{\mu_k})^{\mathfrak{S}_n}_{Y}$. 
Here $Y$ denotes the Young subgroup $\fS_{n_1}\times \fS_{n_2}\times\cdots\times \fS_{n_k}$ of $\fS_n$. 
The following lemma describes the behavior of the first parts of the partitions involved in a non-zero Littlewood-Richardson coefficient. This will be used several times in the following sections. 
\begin{lemma}\label{lem: LR}
Let $\mathcal{LR}(\lambda; \mu_1,\ldots, \mu_k)\neq 0$ then $\lambda_1\leq \sum_{j=1}\6k(\mu_j)_1$. 
\end{lemma}
\begin{proof}
When $k=2$, the statement is a straightforward consequence of the combinatorial description of the Littlewood-Richardson coefficient 
$\mathcal{LR}(\lambda; \mu_1, \mu_2)$, as given in \cite[Section 5.2]{Fulton}. The lemma is then proved by iteration. 
\end{proof}

As in \cite{GL2}, we define $\mathcal{B}_n(t)$ as the set of those partitions of $n$ whose Young diagram fits inside a $t\times t$ square grid, i.e. for $n,t \in \mathbb{N}$, we set
\[ \mathcal{B}_n(t):=\Set{\lambda \in \mathcal{P}(n) | \lambda_1 \leq t, \ l(\lambda)\leq t}. \]

Moreover, for $(n_1,\ldots, n_k)\in\mathcal{C}(n)$ and $A_j\subseteq \mathcal{P}(n_j)$ for all $j\in [1, k]$, we let \[A_1\star A_2\star\cdots\star A_k := \Set{\lambda \in \mathcal{P}(n) | \mathcal{LR}(\lambda; \mu_1,\ldots, \mu_k) > 0,\ \text{for some}\ \mu_1\in A_1, \ldots, \mu_k\in A_k}.\]
It is easy to check that $\star$ is both commutative and associative.
The following lemma was first proved in \cite[Proposition 3.3]{GL2}.
\begin{lemma}\label{lem: B star}
Let $n,n',t,t'\in\mathbb{N}$ be such that $\tfrac{n}{2}<t\le n$ and $\tfrac{n'}{2}<t'\le n'$. Then 
	$$\mathcal{B}_n(t) \star\mathcal{B}_{n'}(t') = \mathcal{B}_{n+n'}(t+t').$$
\end{lemma}

\section{Preliminary results}
In this section we start collecting some results on restriction of characters to Sylow $p$-subgroups. These will be used to prove our main theorems in the second part of the paper. 

Unless otherwise stated, from now on $p$ will always denote a fixed odd prime number.
Let $n\in \mathbb{N}$ and let $n=\sum_{i=1}\6{t}p\6{n_i}$ be its $p$-adic expansion, where $n_1\geq n_2\geq\cdots \geq n_t\geq 0$. 
We define the integer $\alpha_n$ as follows. For powers of $p$ we set $\alpha_1=\alpha_p=0$ and $\alpha_{p^k}=(p^{k-1}-1)/(p-1)$, for $k\geq 2$. For general $n=\sum_{i=1}\6{t}p\6{n_i}$, we set $ \alpha_n = \sum_{i=1}^t \alpha_{p^{n_i}}$.
%
%
%
As shown in Lemma \ref{lem: cd_P} below, $p\6{\alpha_n}$ is the greatest degree of an irreducible character of $P_n$. 
It is interesting to note that $\alpha_n=\nu(\lfloor\frac{n}{p}\rfloor!)$, where $\nu(n)$ denotes the highest power of $p$ dividing $n$. We omit the proof of this statement and we refer the reader to \cite{GVThesis} for the complete calculations. 

As mentioned in the Introduction, we let $\mathrm{Irr}_k(P_n)=\{\theta\in\mathrm{Irr}(P_n)\ |\ \theta(1)=p^k\}$. 
This notation will be kept throughout the article. 
In the following lemma we give a lower bound for the size of the set $\mathrm{Irr}_k(P_n)$. This is certainly far from being attained (in general), but it will be sufficient for our purposes. 
\begin{lemma} \label{lem: neq irrid}
Let $k,t\in \mathbb{N}$ be such that $p^k \in \cd(P_{p^t})$. Then $|\mathrm{Irr}_k(P_{p^t})|\geq p$.
\end{lemma}
\begin{proof}
	We proceed by induction on $t$. 
	If $t=1$ then we know that the statement holds as $\Irr(P_p)=\mathrm{Irr}_0(P_p)$ has size $p$. 
	The elements of $\Irr(P_p)$ are denoted by $\phi_0,\phi_1,\ldots, \phi_{p-1}$, where we conventionally set $\phi_0$ to be the trivial character. 
	Let $t\geq 2$, and let $\psi\in \Irr_k(P_{p^t})$. 
	If $\psi=\lX(\theta; \phi_i)$ for some $\theta\in \Irr_k(P_{p^{t-1}})$ and $i \in [0,p-1]$, then $ \lX(\theta; \phi_j)\in\mathrm{Irr}_k(P_{p^t})$ for all $j\in [0,p-1]$. Hence 
	$|\mathrm{Irr}_k(P_{p^t})|\geq p$. 
	Otherwise $\psi = (\theta_1 \times \cdots \times \theta_p)^{P_{p^t}}$ where $\theta_1, \dots ,\theta_p \in \Irr(P_{p^{t-1}})$ are not all equal. If there exists $x\in[1,p]$ such that $\theta_1(1)\neq \theta_x(1)$ then we define $\eta_1,\ldots,\eta_p\in\mathrm{Irr}_k(P_{p^t})$ as follows. For any $j\in [1,p]$ we let
	$$\eta_j=(\tau_j\times\theta_2\times\cdots\times\theta_p)^{P_{p^t}},$$
	where $\tau_1,\tau_2,\ldots, \tau_p$ are $p$ distinct irreducible characters of $P_{p^{t-1}}$ of degree  $\theta_1(1)$. These exist by inductive hypothesis. 
	On the other hand, if $\theta_1(1)=\theta_x(1)$ for all $x\in [1,p]$ then we let $$\eta_1=(\tau_2\times\cdots \times \tau_2\times\tau_1)^{P_{p^t}},\ \text{and}\ \eta_j=(\tau_1\times\cdots\times\tau_1\times\tau_j)^{P_{p^t}},\ \text{for all}\ j\in [2,p].$$
As before, here we chose $\tau_1,\tau_2,\ldots, \tau_p$ to be $p$ distinct irreducible characters of $P_{p^{t-1}}$ of degree $\theta_1(1)$. These exist by inductive hypothesis. In both cases $\eta_1,\ldots, \eta_p$ are $p$ distinct elements of $\mathrm{Irr}_k(P_{p^t})$. 
Hence $|\mathrm{Irr}_k(P_{p^t})|\geq p$.
\end{proof}

The next lemma shows that $P_n$ has irreducible characters of each degree $1,p,p^2,\ldots, p^{\alpha_n}$.

\begin{lemma} \label{lem: cd_P}
Let $n\in\mathbb{N}$. Then $\cd(P_{n})=\{p\6k\ |\ k\in [0,\alpha_{n}]\}$. 
\end{lemma}
\begin{proof}
	Let us first suppose that $n=p^t$ is a power of $p$ and proceed by induction on $t$.
	The case $t=1$ is trivial, since $P_p$ is cyclic and $\alpha_p=0$.
	If $t\geq 2$, notice that $\alpha_{p^t}=1 +p \alpha_{p^{t-1}}$. Let $k\in [0,\alpha_{p^t} -1]$, and let $q\leq \alpha_{p^{t-1}}$ and $r\in[0,p-1]$ be such that $k=qp+r$.
	If $r=0$, by inductive hypothesis there exists $\phi \in \Irr(P_{p^{t-1}})$ such that $\phi(1)=p^q$. Hence for any $\psi \in \Irr(P_p)$, $\lX (\phi; \psi) \in \Irr(P_{p^t})$ has degree $p^k$.
	If $r>0$, then $q< \alpha_{p^{t-1}}$. By inductive hypothesis, there exist $\phi_1,\dots, \phi_p \in \Irr(P_{p^{t-1}})$ such that $\phi_i(1)=p^{q+1}$ for every $i\in [1,r]$, and $\phi_j(1)=p^q$ for every $j\in [r+1,p]$.
	Hence $(\phi_1 \times \cdots \times \phi_r \times \phi_{r+1} \times \cdots \times \phi_p)^{P_{p^t}} \in \Irr(P_{p^t})$ has degree $p^k$.
	Finally, let $k=\alpha_{p^t}$. By inductive hypothesis and by Lemma \ref{lem: neq irrid}, there exist $\phi_1 , \dots , \phi_p \in \Irr(P_{p^{t-1}})$ not all equal and such that $\phi_i(1)=p^{\alpha_{p^{t-1}}}$ for all $i\in [1,p]$. Hence $(\phi_1 \times \cdots \times \phi_p)^{P_{p^t}} \in \Irr(P_{p^t})$ has degree $p^k$.
	This concludes the proof in the case $n=p^t$, for $t\in \mathbb{N}$.	
	
	The case where $n$ is not a power of $p$ follows easily. Indeed, if $n=\sum_{i=1}\6{t}p\6{n_i}$ is the $p$-adic expansion of $n$ then $P_n \cong P_{p^{n_1}} \times P_{p^{n_2}} \times \cdots \times P_{p^{n_t}}$.
\end{proof}

Let $n,k\in \mathbb{N}$. As we mentioned in the introduction, it is convenient to think of the set $\Omega_n^k$ as a subset of $\mathcal{P}(n)$. More precisely, for $\lambda \in \mathcal{P}(n)$, we will sometimes write $\lambda \in \Omega_n^k$ instead of $\chi^\lambda \in \Omega_n^k $.

The following is an important ingredient when proving statements by induction. For an odd prime $p$ let $\chi$ be a non-linear character of $\fS_n$ and suppose that $\chi_{P_n}$ has an irreducible constituent of degree $p^k$. Then it has at least two distinct irreducible constituents. 

\begin{lemma} \label{lem: neq const}
	Let $n\in\mathbb{N}$ be such that $n\geq p$ and let $\lambda\in\Omega_n^k \smallsetminus\{(n), (1^n)\}$ for some $k\in [0,\alpha_n]$. Then there are at least two distinct irreducible constituents of $(\chi\6\lambda)_{P_n}$ of degree $p\6k$. 
\end{lemma}
\begin{proof}
Let us first suppose that $n=p\6t$, and let us set $P=P_{p^t}$ and $N=N_{\fS_{p^t}}(P_{p^t})$. We observe that the only $N$-invariant irreducible character of $P$ is the trivial one. 
To show this, we let $\mathrm{Irr}_K(P)$ denote the set of $K$-invariant irreducible characters of $P$, for any $K\leq N$. 
Let $H$ be a $p'$-complement of $P$ in $N$. Clearly $\mathrm{Irr}_N(P)=\mathrm{Irr}_H(P)$. On the other hand the set $C_P(H)=\{x\in P\ |\ x\6h=x,\ \text{for all}\ h\in H\}$ consists of the only identity element, by Lemma \ref{lem: elements in N}. Using the Glauberman correspondence \cite[Theorem 13.1]{IBook}, we get that $|\mathrm{Irr}_H(P)|=|\mathrm{Irr}(C_P(H))|=1$. It follows that $\mathrm{Irr}_N(P)=\{1_{P}\}$, as claimed. 

Since $\lambda\notin \{(n), (1^n)\}$, by \cite[Lemma 4.3]{GL2} we know that $(\chi\6\lambda)_{P}$ necessarily admits a non-trivial linear constituent (direct computations show that this holds also in the case $(p, n, \lambda)=(3, 9, (3,3,3))$, which is not covered by the lemma). It follows that for any $k\in\mathbb{N}$ such that  $\lambda\in\Omega_n^k$, we can find a non-trivial $\theta\in\mathrm{Irr}_k(P)$ such that $\theta$ is a constituent of $(\chi\6\lambda)_{P}$. Since $\chi\6\lambda$ is $N$-invariant we deduce that every $N$-conjugate of $\theta$ is a constituent of $\chi\6\lambda$. The statement follows. 
Recalling the structure of $P_n$ described in Section \ref{sec2.2}, we observe that the case where $n$ is not a prime power is an easy consequence of the prime power case. 
\end{proof}

\begin{definition}
	Let $G$ be a finite group and let $H$ be a $p$-subgroup of $G$. Given a character $\theta$ of $G$, we let 
	$\mathrm{cd}(\theta_H)$ be the set of degrees of the irreducible constituents of $\theta_H$. Moreover, we let
	$\partial_H(\theta)$ be the non-negative integer defined as follows: $$\partial_H(\theta)=\mathrm{max}\{k\in\mathbb{N}\ |\ p\6k\in \mathrm{cd}(\theta_H)\} .$$
\end{definition}

\begin{proposition}\label{prop: partial}
	Let $n\in\mathbb{N}$ and let $Y=(\fS_{p\6{n-1}})\6{\times p}\leq \fS_{p\6{n}}$ be such that $B\leq Y$, where $B=(P_{p\6{n-1}})\6{\times p}$ is the base group of $P_{p\6n}$. Let $\lambda\in\mathcal{B}_{p\6n}(p\6{n}-1)$. Then $$\partial_{P_{p\6n}}(\chi\6\lambda)=1+\mathrm{max}\{\partial_{B}(\chi\6{\nu_1}\times\cdots\times\chi\6{\nu_p})\ |\ \chi\6{\nu_1}\times\cdots\times\chi\6{\nu_p}\in\mathrm{Irr}(Y)\ \text{and}\ \mathcal{LR}(\lambda; \nu_1,\ldots, \nu_p)\neq 0\}.$$
\end{proposition}
\begin{proof}
	Let $M=1+\mathrm{max}\{\partial_{B}(\chi\6{\nu_1}\times\cdots\times\chi\6{\nu_p})\ |\ \chi\6{\nu_1}\times\cdots\times\chi\6{\nu_p}\in\mathrm{Irr}(Y)\ \text{and}\ \mathcal{LR}(\lambda; \nu_1,\ldots, \nu_p)\neq 0\}.$
	Let $\mu_1,\dots ,\mu_p \in \mathcal{P}(p^{n-1})$ be such that $\mathcal{LR}(\lambda; \mu_1,\ldots, \mu_p)\neq 0$ and $M=1+\partial_{B}(\chi\6{\mu_1}\times\cdots\times\chi\6{\mu_p})$. 
Since $\lambda\notin\{(n), (1\6n)\}$, we can assume that $\mu_1,\ldots, \mu_p$ are not all in $\{(p\6{n-1}), (1\6{p\6{n-1}})\}$.
	Moreover, let $\phi$ be an irreducible constituent of $ (\chi\6{\mu_1}\times\cdots\times\chi\6{\mu_p})_B$ such that $\phi(1)=p^{M-1}$.
	By Lemma \ref{lem: neq const}, we can take $\phi=\phi_1 \times \cdots \times \phi_p$ with $\phi_1,\dots ,\phi_p \in \Irr(P_{p^{n-1}})$ not all equal. Hence $\phi^{P_{p^n}} \in \Irr(P_{p^n})$, it has degree $p^M$ and $\left[ (\chi^\lambda)_{P_{p^n}}, \phi^{P_{p^n}}  \right] \neq 0$. Thus $p^M\in \cd((\chi^\lambda)_{P_{p^n}})$.
		
	Now suppose for a contradiction that there exists an integer $N>M$ such that $p^N \in \cd((\chi^\lambda)_{P_{p^n}})$. Then there exists $\varphi \in \Irr(P_{p^n})$ such that $\left[ (\chi^\lambda)_{P_{p^n}} , \varphi \right]\neq 0$ and $\varphi(1)=p^N$.
	Let $\phi_1\times\cdots\times\phi_p$ be an irreducible constituent of $\varphi_B$. Hence there exist $\mu_1, \dots ,\mu_p \in \mathcal{P}(p^{n-1})$ such that $\mathcal{LR}(\lambda; \mu_1,\ldots, \mu_p)\neq 0$ and $\left[ \phi_1 \times \cdots \times \phi_p, (\chi^{\mu_1}\times \cdots \times \chi^{\mu_p} )_B \right] \neq 0$. We have that
	\[ M> \partial_{B}(\chi\6{\mu_1}\times\cdots\times\chi\6{\mu_p}) \geq N-1 ,\]
	since the degree of $\phi_1 \times \cdots \times \phi_p$ is either $p^N$ or $p^{N-1}$.
	Hence $M<N<M+1$, which is a contradiction.
\end{proof}

\begin{proposition}\label{prop: partial arbitrary n}
	Let $n$ be a natural number and let $n=\sum_{i=1}\6{t}p\6{n_i}$ be the $p$-adic expansion of $n$, where $n_1\geq n_2\geq\cdots \geq n_t\geq 0$. 
	Let $Y=\fS_{p\6{n_1}}\times\fS_{p\6{n_2}}\times  \cdots\times \fS_{p\6{n_t}}$ be such that $P_n \leq Y \leq \fS_n$, and 
	let $\lambda$ be a partition of $n$. 
	Then
	$$\partial_{P_n}(\chi\6\lambda)=\mathrm{max}\{\partial_{P_n}(\chi\6{\mu_1}\times\cdots\times\chi\6{\mu_t})\ |\ \chi\6{\mu_1}\times\cdots\times\chi\6{\mu_t}\in\mathrm{Irr}(Y)\ \text{and}\ \mathcal{LR}(\lambda; \mu_1,\ldots, \mu_t)\neq 0\}.$$
\end{proposition}
\begin{proof}
	Since $P_n = P_{p^{n_1}} \times P_{p^{n_2}} \times \cdots \times P_{p^{n_t}}\leq Y$, the statement follows.
\end{proof}

\begin{lemma}\label{lem: Omega star-p-times}
Let $n\in\mathbb{N}_{\geq 2}$, let $k\in [2,\alpha_{p\6n}]$ and let $(a_1,\ldots, a_p)\in\mathcal{C}(k-1)$ be such that $a_i\in [0,\alpha_{p\6{n-1}}]$, for all $i\in [1,p]$. 
Then $$\Omega_{p^{n-1}}\6{a_1}\star\Omega_{p^{n-1}}\6{a_2}\star\cdots\star\Omega_{p^{n-1}}\6{a_p}\subseteq \Omega_{p\6n}\6{k}.$$
\end{lemma}
\begin{proof}
To ease the notation we let $q=p\6{n-1}$.
If $\lambda \in \Omega_{q}\6{a_1}\star\Omega_{q}\6{a_2}\star\cdots\star\Omega_{q}\6{a_p}$, by definition there exists an irreducible constituent $\chi^{\mu_1} \times \dots \times \chi^{\mu_p}$ of $(\chi^\lambda)_{(\fS_{q})^{\times p}}$ such that $\mu_i \in \Omega_{q}\6{a_i}$ for all $i\in [1,p]$.
Hence for every $i\in [1,p]$ there exists an irreducible constituent $\phi_i$ of $(\chi\6{\mu_i})_{P_{q}}$ such that $\phi_i(1)=p^{a_i}$.
Since $k\geq 2$, there exists $j\in [1,p]$ such that $a_j\geq 1$. Hence $\mu_j\notin\{(q), (1\6{q})\}$. Thus,
by Lemma \ref{lem: neq const} we can assume that $\phi_1,\dots ,\phi_p$ are not all equal. 
It follows that $(\phi_1 \times \cdots \times \phi_p)^{P_{p^n}}$ is an irreducible constituent of $(\chi^\lambda)_{P_{p\6n}}$ of degree equal to $p^k$. Hence $\lambda\in \Omega_{p\6n}\6{k}.$
\end{proof}

\begin{lemma}\label{lem: Omega star-R-times}
Let $n\in\mathbb{N}_{\geq 2}$ and let $n=\sum_{i=1}\6{t}p\6{n_i}$ be its $p$-adic expansion, where $n_1\geq n_2\geq\cdots \geq n_t\geq 0$.
Let $k\in [1,\alpha_{n}]$ and let $(a_1,\ldots, a_t)\in\mathcal{C}(k)$ be such that $a_i\in [0,\alpha_{p\6{n_i}}]$, for all $i\in [1,t]$.
Then $$\Omega_{p\6{n_1}}\6{a_1}\star\Omega_{p\6{n_2}}\6{a_2}\star\cdots\star\Omega_{p\6{n_t}}\6{a_t}\subseteq \Omega_{p^n}\6{k}.$$
\end{lemma}
\begin{proof}
	Recall that $P_n \cong P_{p^{n_1}} \times P_{p^{n_2}} \times \cdots \times P_{p^{n_t}}$ and let $\lambda \in \Omega_{p\6{n_1}}\6{a_1}\star\Omega_{p\6{n_2}}\6{a_2}\star\cdots\star\Omega_{p\6{n_t}}\6{a_t} $. By definition, for every $i\in [1,t]$ there exists $\phi_i \in \Irr(P_{p^{n_i}})$ with $\phi_i(1)=p\6{a_i}$, such that $\phi_1 \times \cdots \times \phi_t$ is an irreducible constituent of $(\chi^\lambda)_{P_n}$ of degree $p^k$. Hence $\lambda\in \Omega_{p^n}\6{k}.$
\end{proof}

\section{The prime power case}

The aim of this section is to completely describe the sets $\Omega_{p\6n}\6k$ for all odd primes $p$, all natural numbers $n$ and all $k\in [0,\alpha_{p\6n}]$. 
We remind the reader that from \cite[Theorem 3.1]{GN}, we know that $\Omega_{p^n}^0 =\mathcal{B}_{p^n}(p^n)$, for all $n\in\mathbb{N}$. Equivalently, every irreducible character of $\fS_{p^n}$ admits a linear constituent on restriction to a Sylow $p$-subgroup. 
This result will be used frequently, with no further reference. 
We start by analysing the cases where $k\in \{1,2\}$. In the next lemma we show that for $j=1,2,$ every non-linear character of $\fS_{p^n}$ affords an irreducible constituent of degree $p^j$ on restriction to a Sylow $p$-subgroup $P_{p^n}$, as long as $P_{p^n}$ has an irreducible character of degree $p^j$.

\begin{lemma}\label{lem: 4.0}
Let $p$ be an odd prime, $n\in\mathbb{N}$ and $k\in [1,\alpha_{p\6n}]\cap\{1,2\}$. Then $\Omega_{p\6n}\6k=\mathcal{B}_{p\6n}(p\6n-1)$.
\end{lemma}
\begin{proof}
Let $k=1$. Then necessarily $n\geq 2$. We first observe that clearly $\Omega_{p\6n}\61\subseteq \mathcal{B}_{p^n}(p^n-1)$.
On the other hand, if $\lambda\in\mathcal{B}_{p^n}(p^n-1)$, then there exist $\mu_1\in\mathcal{B}_{p^{n-1}}(p^{n-1}-1)$ and $\mu_2,\ldots, \mu_p\in\mathcal{P}(p\6{n-1})$ such that $\mathcal{LR}(\lambda; \mu_1,\ldots, \mu_p)\neq 0$. Using Lemma \ref{lem: neq const} we deduce that $(\chi\6{\mu_1})_{P_{p\6{n-1}}}$ admits two distinct linear constituents. Therefore, there exists $\phi_1,\ldots, \phi_p\in\mathrm{Lin}(P_{p\6{n-1}})$ not all equal and such that 
$\phi_i$ is a constituent of $(\chi\6{\mu_i})_{P_{p\6{n-1}}}$, for all $i\in [1,p]$. It follows that $(\phi_1\times\cdots\times\phi_p)\6{P_{p\6n}}$ is an irreducible constituent of $(\chi\6{\lambda})_{P_{p\6{n}}}$ of degree $p$. We conclude that $\lambda\in\Omega_{p\6n}\61$ and hence that $\Omega_{p\6n}\61=\mathcal{B}_{p\6n}(p\6n-1)$.

Let $k=2$. Then necessarily $n\geq 3$. It is clear that $\Omega_{p\6n}\62\subseteq \mathcal{B}_{p^n}(p^n-1)$. On the other hand, if $\lambda\in\mathcal{B}_{p^n}(p^n-1)$, then there exist $\mu_1\in\mathcal{B}_{p^{n-1}}(p^{n-1}-1)$ and $\mu_2,\ldots, \mu_p\in\mathcal{P}(p\6{n-1})$ such that $\mathcal{LR}(\lambda; \mu_1,\ldots, \mu_p)\neq 0$. We can now argue exactly as above to deduce that $(\chi\6\lambda)_{P_{p\6n}}$ admits an irreducible constituent $\theta$ of the form 
$\theta=(\psi\times\phi_1\times\cdots\times\phi_{p-1})\6{P_{p\6n}}$, where $\psi\in\mathrm{Irr}_1(P_{p\6{n-1}})$ and $\phi_1,\ldots, \phi_{p-1}\in\mathrm{Lin}(P_{p\6{n-1}})$. Hence $\theta(1)=p\62$, $\lambda\in\Omega_{p\6n}\62$ and therefore we have that $\Omega_{p\6n}\62=\mathcal{B}_{p\6n}(p\6n-1)$.
\end{proof}

Lemma \ref{lem: 4.0} is a special case of the following more general result.

\begin{theorem}\label{thm: boxes}
Let $n\in\mathbb{N}$ and let $k\in [0,\alpha_{p\6n}]$. Then there exists $t_n\6k\in [\frac{p\6n+1}{2}, p\6n]$ such that
$\Omega_{p^n}\6k=\mathcal{B}_{p\6n}(t_n\6k).$ Moreover, if $k\in [0,\alpha_{p\6n}-1]$, then $t_n\6{k+1}\in\{t_n\6{k}-1, t_n\6{k}\}.$
\end{theorem}
\begin{proof}	
We proceed by induction on $n$. If $n=1$, then $\alpha_p=0$ and $\Omega_{p}^0=\mathcal{P}(p)$. 
If $n\geq 2$, we assume that the statement holds for $n-1$. 
If $k=0$ then by \cite[Theorem 3.1]{GN}, $\Omega_{p^n}^0 =\mathcal{B}_{p^n}(p^n)$, and $t_n^0=p^n$. Moreover, by Lemma \ref{lem: 4.0} we know that $t_n\61=p\6n-1=t_n^0-1$, as required. 
The case $k=1$ is completely treated by Lemma  \ref{lem: 4.0}. In fact, we know that $\Omega_{p\6n}\61=\mathcal{B}_{p^n}(p^n-1)$ and that $t_n\62=p\6n-1=t_n^1$, as required. 
We can now suppose that $k\geq 2$.
We define 
$$\mathcal{L}(k-1)=\{(j_1,\dots ,j_p)\in \mathcal{C}(k-1)\ |\ j_i \in [0,\alpha_{p^{n-1}}] \text{   for all } i\in [1,p]\}.$$
Moreover, we set 
\[M= \max \Set{t_{n-1}^{j_1}+\cdots +t_{n-1}^{j_p} | (j_1,\dots ,j_p)\in \mathcal{L}(k-1)} .\]
Notice that for any $j\in [0,\alpha_{p\6{n-1}}]$, the value $t_{n-1}^{j}$ is well-defined by induction as the integer such that $\Omega_{p\6{n-1}}\6j=\mathcal{B}_{p\6{n-1}}(t_{n-1}^{j})$.
We claim that $M=t_n\6k$. In other words, we want to prove that $\Omega_{p^n}^k =\mathcal{B}_{p^n}(M) $.
Let $(j_1, \dots, j_p) \in \mathcal{L}(k-1)$ be such that $M=t_{n-1}^{j_1}+\cdots +t_{n-1}^{j_p}$.
	By inductive hypothesis and by Lemmas \ref{lem: B star} and \ref{lem: Omega star-p-times}, we have that $$\mathcal{B}_{p^n}(M)=\mathcal{B}_{p^{n-1}}(t_{n-1}^{j_1})\star \cdots \star \mathcal{B}_{p^{n-1}}(t_{n-1}^{j_p}) = \Omega_{p\6{n-1}}^{j_1} \star \cdots \star \Omega_{p^{n-1}}^{j_p} \subseteq \Omega_{p\6{n}}^k .$$

For the opposite inclusion, suppose for a contradiction that $\lambda \in \Omega_{p^n}^k \smallsetminus \mathcal{B}_{p^n}(M)$. Since $p$ is odd, we have that $\Omega_{p^n}^k$ is closed under conjugation of partitions. Hence, we can assume that $\lambda_1 \geq M+1$. Since $\lambda\in\Omega_{p^n}^k$, there exists an irreducible constituent $\theta$ of $(\chi^\lambda)_{P_{p^n}}$ with $\theta(1)=p^k$.

\noindent $\bullet$ If $\theta= ( \phi_1 \times \cdots \times \phi_p )^{P_{p^n}}$ with $\phi_1, \ldots, \phi_p \in \Irr(P_{p^{n-1}})$ not all equal, then there exists $(j_1,\ldots, j_p)\in \mathcal{L}(k-1)$ such that 
$\phi_i(1)=p^{j_i}$ for all $ i\in [1,p]$. Then, for every $ i\in [1,p]$ there exists an irreducible constituent $\chi^{\mu_i}$ of $(\phi_i)^{\fS_{p^{n-1}}}$ such that $\left[ \chi^{\mu_1} \times \cdots \times \chi^{\mu_p} , (\chi^\lambda)_{(\fS_{p^{n-1}})^{\times p}} \right] \neq 0$.
		Hence using the inductive hypothesis, we have that $\mu_i \in \Omega_{p^{n-1}}^{j_i}=\mathcal{B}_{p^{n-1}}(t_{n-1}^{j_i})$, for all $ i\in [1,p]$. 
		Hence
		\[M \geq t_{n-1}^{j_1}+\cdots +t_{n-1}^{j_p} \geq \lambda_1 \geq M+1 ,\]
		where the first inequality holds by definition of $M$ and the second one by Lemma \ref{lem: LR}. This is a contradiction.
		
\noindent $\bullet$ On the other hand, if $\theta =\lX \left( \phi; \psi \right)$ for some $\phi \in \Irr(P_{p^{n-1}})$ and $\psi \in \Irr(P_p)$, then, $\phi(1)=p^{\frac{k}{p}}$ and there exist $\mu_1, \dots , \mu_p \in \Omega_{p^{n-1}}^{\frac{k}{p}}$ such that $\mathcal{LR}(\lambda; \mu_1, \dots , \mu_p) \neq 0$. Hence, using the inductive hypothesis we have that
$$\lambda \in \left( \Omega_{p^{n-1}}^{\frac{k}{p}} \right)^{\star p}=\left( \mathcal{B}_{p\6{n-1}}(t_{n-1}\6{\frac{k}{p}}) \right)^{\star p}.$$ Here we denoted by $A^{\star p}$ the $p$-fold $\star$-product $A \star \cdots \star A$.
		By inductive hypothesis we also know that $t_{n-1}^{\frac{k}{p}} \in \Set{t_{n-1}^{\frac{k}{p} -1}-1, t_{n-1}^{\frac{k}{p} -1} }$. Using Lemma \ref{lem: LR} we obtain that
		\[M+1 \leq \lambda_1 \leq p t_{n-1}^{\frac{k}{p}} \leq (p-1)t_{n-1}^{\frac{k}{p} } +t_{n-1}^{\frac{k}{p} -1} \leq M . \]
	This is a contradiction. Notice that the last inequality above follows from the definition of $M$, as $(\frac{k}{p},\ldots, \frac{k}{p}, \frac{k}{p}-1)\in\mathcal{L}(k-1)$.

For $k\in [2, \alpha_{p^n} -1]$, what we have proved so far is summarised here. 
\[\begin{split} \Omega_{p^n}^k =\mathcal{B}_{p^n}(T) \mbox{, with } &T= \max \Set{ t_{n-1}^{j_1}+\cdots +t_{n-1}^{j_p} | 
	(j_1,\dots , j_p) \in \mathcal{L}(k-1)}
\\ \Omega_{p^n}^{k+1}=\mathcal{B}_{p^n}(V) \mbox{, with } &V=\max \Set{t_{n-1}^{h_1}+\cdots +t_{n-1}^{h_p} | 	(h_1,\dots , h_p) \in \mathcal{L}(k)} .	\end{split}\]
Let $(j_1,\dots ,j_p)\in \mathcal{L}(k-1)$ be such that $T=t_{n-1}^{j_1}+\cdots +t_{n-1}^{j_p} $. Without loss of generality, we can assume that $j_1 < \alpha_{p^{n-1}}$. Then $(j_1 +1,j_2,\dots ,j_p) \in \mathcal{L}(k)$. By inductive hypothesis we know that $t_{n-1}^{j_1 +1} \in \Set{t_{n-1}^{j_1}-1 , t_{n-1}^{j_1}}$. Hence
\begin{equation} \label{V>T}
V\geq t_{n-1}^{j_1 +1}+t_{n-1}^{j_2}+\cdots +t_{n-1}^{j_p} \in \Set{T-1,T}. \end{equation}
On the other hand, let $(h_1, \dots ,h_p)\in \mathcal{L}(k)$ be such that $V=t_{n-1}^{h_1}+\cdots +t_{n-1}^{h_p}$. Since $k\geq 2$, without loss of generality we can assume that $h_1>0$. Then $(h_1 -1,h_2,\dots , h_p )\in \mathcal{L}(k-1)$. Thus, as above:
\begin{equation} \label{V<T} 
V=t_{n-1}^{h_1}+\cdots + t_{n-1}^{h_p} \leq T ,\end{equation}
since $t_{n-1}^{h_1} \in \Set{t_{n-1}^{h_1 -1}-1, t_{n-1}^{h_1-1}}$. Inequalities (\ref{V>T}) and (\ref{V<T}) imply that $V\in \Set{T-1, T}$.
\end{proof}

We refer the reader to the second part of Example \ref{ex: 3^n} for a description of the key steps of the proof of Theorem \ref{thm: boxes} in a small concrete instance. 
The following definitions may seem artificial but are crucial for determining the exact value of $t_n\6k$ for all $n,k\in\mathbb{N}$.

\begin{definition}\label{def: Ax}
Let $n\in\mathbb{N}_{\geq 2}$ and let $x\in [1,p\6{n-2}]$. We define the integers $m_x$ and $\ell(n,x)$ as follows:
 $$m_x=\mathrm{min}\{m\ |\ x\leq p\6{m-2}\},\ \text{and}\ \ell(n,x)=n-m_x+1.$$
Notice that $\sum_{x=1}\6{p\6{n-2}}\ell(n,x)=\alpha_{p\6n}$ (this is proved in Lemma \ref{lem: 3x} below). For $x\in [1,p\6{n-2}]$ we let 
$$A_x=\left[\sum_{j=1}\6{x-1}\ell(n,j)+1, \sum_{j=1}\6{x}\ell(n,j)\right].$$
We observe that $\{A_1, A_2,\ldots, A_{p\6{n-2}}\}$ is a partition of $[1,\alpha_{p\6{n}}]$ and that $|A_x|=\ell(n,x)$ for all $x\in [1,p\6{n-2}]$.
We refer the reader to Example \ref{ex: 3^n} for a description of these objects in a specific setting.
\end{definition}

For the convenience of the reader we give a more informal explanation of Definition \ref{def: Ax} above. For fixed $n\geq 2$, we define an increasing sequence $0=a_0<a_1<a_2<\cdots<a_{p^{n-2}}=\alpha_{p\6{n}}$ as follows. First $a_1=n-1$. Then $a_i-a_{i-1}=n-2$, for $i=2,\ldots, p$. Next $a_i-a_{i-1}=n-3$, for $i=p+1,p+2,\ldots, p^2$. Continue in this manner, we find that $a_i-a_{i-1}=1$, for $i=p^{n-3}+1,\ldots, p^{n-2}$. Now set $A_i:=(a_{i-1}, a_i]$, for $i=1,\ldots, p^{n-2}$. Then $\{A_1,A_2,\ldots, A_{p^{n-2}}\}$ is clearly a partition of $[1,p\6{n-2}]$.

\begin{lemma}\label{lem: 3x}
With the notation introduced in Definition \ref{def: Ax}, we have that $\sum_{x=1}\6{p\6{n-2}}\ell(n,x)=\alpha_{p\6n}$.
\end{lemma}
\begin{proof}
If $n=2$, then $\ell(2,1)=1=\alpha_{p\62}$. 
	Let $n\geq 3$ and $i\in [0,n-3]$, then for every $x\in [p^i +1 , p^{i+1}]$, $m_x =i+3$ and $\ell(n,x)=n-i-2$. Hence
	\[\begin{split}
	\sum_{x=1}\6{p\6{n-2}}\ell(n,x) &= \ell(n,1) + \sum_{i=0}^{n-3} \sum_{x= p^i +1}^{p^{i+1}} \ell(n,x) = (n-1) + \sum_{i=0}^{n-3} p^i(p-1) (n-i-2) \\
	&= (n-1)-(n-2) + \left( \sum_{i=1}^{n-3} p^i [(n-i-1)-(n-i-2)] \right) + p^{n-2}[n-(n-3)-2] \\
	&=1+ \left(\sum_{i=1}^{n-3} p^i \right) +p^{n-2} =\alpha_{p^n} .
	\end{split}\]
\end{proof}

The following technical lemma will be useful to prove Theorem \ref{thm: p-power}.

\begin{lemma}\label{lem: ell(n,j)}
Let $n\in\mathbb{N}\geq 2$ and $p$ be an odd prime. 
If $x=pa+r$, for some $r\in [0,p-1]$ and $a\in\mathbb{N}$, then 
$$p\cdot \sum_{j=1}\6a\ell(n-1, j)+ r\cdot\ell(n-1,a+1)
= \sum_{j=1}\6{x}\ell(n,j)-1.$$
\end{lemma}
\begin{proof}
	Notice that $\ell(n,1)=n-1$ and if $y\in [2,p]$, $\ell(n,y)=n-2$. Thus $$\sum_{y=1}^p \ell(n,y) =p \ell(n-1,1) +1.$$
	
	Moreover, for $j\in \mathbb{N}$ we have that $$\sum_{y=jp+1}^{jp+p} \ell(n,y) = p \ell(n-1,j+1).$$ This follows by observing that $\ell(n,y)=\ell(n-1,j+1)$, for all $y\in[jp+1, jp+p]$.
	
	Using these facts, we deduce that
	\[ \begin{split} \sum_{j=1}\6{x}\ell(n,j) &= \sum_{j=1}^p \ell(n,j) + \sum_{j=1}^{a-1} \sum_{y=jp+1}^{jp+p} \ell(n,y) +\sum_{i=1}^r \ell(n,ap+i) \\
	&= 1+p \ell(n-1,1) +p\sum_{j=1}^{a-1} \ell(n-1,j+1) + r \ell(n-1,a+1)   .\end{split} \]
\end{proof}

The main result of this section shows that if $p^k$ is a character degree of $P_{p^n}$, then the partitions of $p^n$ whose corresponding irreducible character admit a constituent of degree $p^k$ on restriction to $P_{p^n}$ are precisely those which fit inside a square of length $p^n-x$, where $k\in A_x$ determines $x$. 

\begin{theorem}\label{thm: p-power}
Let $n\geq 2$, $k\in [1,\alpha_{p^n}]$ and let $x\in [1,p\6{n-2}]$ be such that $k\in A_x$. Then $\Omega_{p^n}\6k=\mathcal{B}_{p\6n}(p\6n-x).$
\end{theorem}
\begin{proof}
We proceed by induction on $n$: if $n=2$ then $\alpha_{p^2}=1$ and necessarily $k=1$ as $A_1=\{1\}$. By Lemma \ref{lem: 4.0}, we have that $\Omega_{p^2}\61=\mathcal{B}_{p\62}(p\62-1)$, as required. 
If $n\geq 3$, we proceed by induction on the parameter $x\in [1,p\6{n-2}]$. For $x=1$, we want to show that for every $k\in A_1=[1,\ell(n,1)]$ we have that $\Omega_{p^n}\6k=\mathcal{B}_{p\6n}(p\6n-1).$ Using Theorem \ref{thm: boxes} and Lemma \ref{lem: 4.0}, we know that 
$$\Omega_{p^n}\6{\ell(n,1)}\subseteq \Omega_{p^n}\6k\subseteq \Omega_{p^n}\61=\mathcal{B}_{p\6n}(p\6n-1).$$
Hence,  it is enough to show that $\Omega_{p^n}\6{\ell(n,1)}=\mathcal{B}_{p\6n}(p\6n-1)$. 
Since $\ell(n,1)=\ell(n-1,1)+1$, we use Lemma \ref{lem: Omega star-p-times}, the inductive hypothesis on $n$ and \cite[Theorem 3.1]{GN}, to deduce that 
$$\Omega_{p^n}\6{\ell(n,1)}\supseteq \Omega_{p^n}\6{\ell(n-1,1)}\star \big(\Omega_{p^n}\6{0}\big)\6{\star p-1}=\mathcal{B}_{p^n}(p^n -1) \star \big( \mathcal{B}_{p^n}(p^n) \big)^{\star p-1}.$$
Using Lemma \ref{lem: B star} we conclude that 
$\mathcal{B}_{p\6{n}}(p\6n-1)\subseteq \Omega_{p^n}\6{\ell(n,1)}$ and therefore that $\mathcal{B}_{p\6{n}}(p\6n-1)= \Omega_{p^n}\6{\ell(n,1)}$. 

Let us now suppose that $x\geq 2$ and that $k\in A_x$. To ease the notation, for any $y\in [1,p\6{n-2}]$ we let $f_n(y)=\sum_{j=1}\6{y}\ell(n,j)$. With this notation we have that $A_x=[f_n(x-1)+1, f_n(x)]$. Using Theorem \ref{thm: boxes} and arguing exactly as above, we observe that in order to show that $\Omega_{p^n}\6k=\mathcal{B}_{p\6n}(p\6n-x)$, it is enough to prove that $$(1)\ \Omega_{p^n}\6{f_n(x-1)+1}=\mathcal{B}_{p\6n}(p\6n-x)\ \text{and that}\ (2)\ \Omega_{p^n}\6{f_n(x)}=\mathcal{B}_{p\6n}(p\6n-x).$$

To prove $(1)$, we start by observing that by inductive hypothesis we know that the statement holds for any $j\in A_{x-1}$. In particular we have that $\Omega_{p^n}\6{f_n(x-1)}=\mathcal{B}_{p\6n}(p\6n-(x-1))$. By Theorem \ref{thm: boxes} it follows that $\Omega_{p^n}\6{f_n(x-1)+1}=\mathcal{B}_{p\6n}(T)$, for some $T\in\{p\6n-x, p\6n-(x-1)\}$. It is therefore enough to show that $\lambda=(p\6n-(x-1), x-1)\notin\Omega_{p^n}\6{f_n(x-1)+1}$. 
Let $\mu_1,\ldots, \mu_p\in\mathcal{P}(p\6{n-1})$ be such that $\mathcal{LR}(\lambda; \mu_1,\dots, \mu_p)\neq 0$. By Lemma \ref{lem: LR} for every $i\in [1,p]$, there exists $a_i\in\mathbb{N}$ such that $(\mu_i)_1=p\6{n-1}-a_i$ and such that $\sum_{j=1}\6pa_j\leq x-1$. In particular,  for every $i\in [1,p]$ we have that 
$$\mu_i\in \mathcal{B}_{p\6{n-1}}(p\6{n-1}-a_i)\smallsetminus \mathcal{B}_{p\6{n-1}}(p\6{n-1}-(a_i+1))=
\Omega_{p^{n-1}}\6{f_{n-1}(a_i)}\smallsetminus\Omega_{p^{n-1}}\6{f_{n-1}(a_i)+1},$$
where the equality is guaranteed by the inductive hypothesis on $n$. 

Let $B=\left( P_{p^{n-1}} \right)^{\times p}$ be the base group of $P_{p^n}$ and 
let $Y=(\fS_{p\6{n-1}})\6{\times p}\leq \fS_{p\6n}$ be such that $B\leq Y$.  Let $\eta=\chi\6{\mu_1}\times\cdots\times\chi\6{\mu_p}\in\mathrm{Irr}(Y)$ and let $x-1=ap+r$, for some $a\in\mathbb{N}$ and $r\in [0,p-1]$. 
We observe that 
\begin{eqnarray*}
\partial_B(\eta) &=& \sum_{j=1}\6pf_{n-1}(a_j)\ =\ \sum_{j=1}\6p\sum_{i=1}\6{a_j}\ell(n-1, i)\\
&\leq & p\cdot \big(\sum_{j=1}\6a\ell(n-1, j)\big)+ r\cdot\ell(n-1,a+1)\\
&=& \sum_{j=1}\6{x-1}\ell(n,j)-1\ =\ f_n(x-1)-1.
\end{eqnarray*}
Here, the inequality follows immediately by observing that $\ell(n-1,s)\geq \ell(n-1,s+1)$ for all $s\in\mathbb{N}$. On the other hand, the third equality holds by Lemma \ref{lem: ell(n,j)}.
Using Proposition \ref{prop: partial}, we deduce that $\partial_{P_{p\6n}}(\chi\6\lambda)\leq f_n(x-1)$. 
It follows that $\lambda\notin\Omega_{p^n}\6{f_n(x-1)+1}$, as desired. 

To prove $(2)$, we recall that by $(1)$ above we have that $\Omega_{p^n}\6{f_n(x-1)+1}=\mathcal{B}_{p\6{n}}(p\6{n}-x)$. Hence, Theorem \ref{thm: boxes} implies that $\Omega_{p^n}\6{f_n(x)}\subseteq\mathcal{B}_{p\6{n}}(p\6{n}-x)$. On the other hand, writing $x=ap+r$ for some $a\in\mathbb{N}$ and $r\in [0,p-1]$, and
using Lemma \ref{lem: ell(n,j)}, we have that: 
 \begin{eqnarray*}
\Omega_{p^n}\6{f_n(x)} &=& \Omega_{p^n}\6{1+p\cdot \big(\sum_{j=1}\6a\ell(n-1, j)\big)+ r\cdot\ell(n-1,a+1)}\\
&\supseteq & \big(\Omega_{p^{n-1}}\6{f_{n-1}(a+1)}\big)\6{\star r}\star\big(\Omega_{p^{n-1}}\6{f_{n-1}(a)}\big)\6{\star p-r}\\
&=& \big(\mathcal{B}_{p\6{n-1}}(p\6{n-1}-(a+1))\big)\6{\star r}\star\big(\mathcal{B}_{p\6{n-1}}(p\6{n-1}-a)\big)\6{\star p-r}\\
&=& \mathcal{B}_{p\6{n}}(p\6{n}-x).
\end{eqnarray*}
Here the first inclusion follows from Lemma \ref{lem: Omega star-p-times}. The second equality holds by inductive hypothesis. Finally, the last equality is given by Lemma \ref{lem: B star}.
The proof is complete. 
\end{proof}

In the following corollary we collect a number of facts useful to have a better understanding of the structure of the sets $\Omega_{p\6n}\6k$ for all $n\in\mathbb{N}$ and all $k\in [0,\alpha_{p\6n}]$. 

\begin{corollary} \label{cor: p^n}
Let  $n\in\mathbb{N}$ and let $1\leq k<t\leq\alpha_{p\6n}$. The following hold. 
\begin{itemize}
\item[(i)] $\mathcal{B}_{p^n}(p^n - p^{n-2}) = \Omega_{p^n}^{\alpha_{p^n}} \subseteq   \Omega_{p\6n}\6t\subseteq \Omega_{p\6n}\6k $.
\item[(ii)] $\Omega_{p\6n}\6k=\Omega_{p\6n}\6t$ if, and only if, there exists $x\in [1,p\6{n-2}]$ such that $k,t\in A_x$. 
\item[(iii)] Given $x\in [1,p\6{n-2}]$ we have that $|\{k\in [1,\alpha_{p\6n}]\ |\ \Omega_{p\6n}\6k=\mathcal{B}(p\6n-x)\}|=\ell(n,x)$. 
\end{itemize}
\end{corollary}
\begin{proof}
	Recalling that $|A_x|=\ell(n,x) $ for every $x\in [1,p\6{n-2}]$, (i), (ii) and (iii) follow immediately by Theorem \ref{thm: p-power}.
\end{proof}

We find particularly surprising that a partition of $p^n$ whose character admits an irreducible constituent of degree $p^k$ on restriction to $P_{p^n}$ also admits a constituent of degree $p^j$, for any $j\in\{0,1,\ldots, k-1\}$. Moreover, the partitions whose character admit a constituent of maximal possible degree $p^{\alpha_{p^n}}$ are precisely those which fit inside a square of side $p^n-p^{n-2}$.

\begin{example} \label{ex: 3^n}
	Let $p=3$ and fix $n=4$. Following the notation introduced in Definition \ref{def: Ax}, we have $3\6{4-2}=9$ and 
	$ \ell(4,1)= 3,\  \ell(4,2)=\ell(4,3)=2,\  \ell(4,4)=\cdots =\ell(4,9)=1 $.
	Hence
	\[ A_1=\Set{1,2,3},\ A_2=\Set{4,5},\ A_3=\Set{6,7},\ A_4=\Set{8},\ A_5=\Set{9}, \dots , A_9=\Set{13} . \]
	Observe that $\Set{A_1, \dots , A_9}$ is a partition of $[1, \alpha_{3^4}]=[1,13]$, as required.
	Using Theorem \ref{thm: p-power}, we have a complete description of $\Omega_{3\64}\6k$, for all $k\in [1,13]$. In particular, we have
	\[\Omega_{3^4}^1=\Omega_{3^4}^2=\Omega_{3^4}^3=\mathcal{B}_{3^4}(3^4-1) ,\ \Omega_{3^4}^4=\Omega_{3^4}^5=\mathcal{B}_{3^4}(3^4-2) ,\ \Omega_{3^4}^6=\Omega_{3^4}^7=\mathcal{B}_{3^4}(3^4-3), \] 
	\[\Omega_{3^4}^8=\mathcal{B}_{3^4}(3^4-4) ,\ \Omega_{3^4}^9=\mathcal{B}_{3^4}(3^4-5) , \dots , \Omega_{3^4}^{13}=\mathcal{B}_{3^4}(3^4-9) .\]
These sets are recorded in the fourth column of Table \ref{table: p=3}. 
	
We use the second part of this example to illustrate a key step of the proof of Theorem  \ref{thm: boxes}. 
Let $n=k=4$. We wish to compute $t_4\64$. 
Following the notation introduced in the proof of Theorem \ref{thm: boxes} we have that
	\[\mathcal{L}(3)=\{(j_1,j_2,j_3)\in \mathcal{C}(3)\ |\ j_i\in [0, \alpha_{3^3}]=[0,4],\text{ for all }i\in [1,3]\}=\{(3,0,0),(2,1,0),(1,1,1)\}.\]
Working by induction we can assume that we know the values $t_3^j$ for every $j\in[0,4]$. This can be comfortably read off the third column of Table \ref{table: p=3}. We set
	\[M=\max \Set{t_3^3+t_3^0+t_3^0, t_3^2+t_3^1+t_3^0, t_3^1+t_3^1+t_3^1}=\max \Set{3^4-2,3^4-2,3^4-3}=3^4-2.\]
We conclude that $t_4\64=M=3\64-2$. 
%
%
%
\end{example}

	\begin{table}[ht]
	\caption{Let $p=3$. According to Theorem \ref{thm: p-power}, the structure of $\Omega_{p^n}^k$ is recorded in the entry corresponding to row $k$ and column $n$.} \label{table: p=3}
	\centering
	\begin{tabular}{ c || c | c | c | c }
		$\Omega_{p^n}^k$ &$n=1$ &$n=2$ &$n=3$ &$n=4$  \\[0.1ex]
		\hline
		\ & \ & \ & \ & \ \\
		$k=0$ & $\mathcal{B}_{3}(3)$ & $\mathcal{B}_{3^2}(3^2)$ & $\mathcal{B}_{3^3}(3^3)$ & $\mathcal{B}_{3^4}(3^4)$ \\
		$k=1$ & $\emptyset$ & $\mathcal{B}_{3^2}(3^2 -1)$ & $\mathcal{B}_{3^3}(3^3 -1)$ & $\mathcal{B}_{3^4}(3^4 -1)$ \\ 
		$k=2$ & $\emptyset$ & $\emptyset$ & $\mathcal{B}_{3^3}(3^3 -1)$ & $\mathcal{B}_{3^4}(3^4 -1)$ \\ 
		$k=3$ & $\emptyset$ & $\emptyset$ & $\mathcal{B}_{3^3}(3^3 -2)$ & $\mathcal{B}_{3^4}(3^4 -1)$ \\ 
		$k=4$ & $\emptyset$ & $\emptyset$ & $\mathcal{B}_{3^3}(3^3 -3)$ & $\mathcal{B}_{3^4}(3^4 -2)$ \\ 
		$k=5$ & $\emptyset$ & $\emptyset$ & $\emptyset$ & $\mathcal{B}_{3^4}(3^4 -2)$ \\ 
		$k=6$ & $\emptyset$ & $\emptyset$ & $\emptyset$ & $\mathcal{B}_{3^4}(3^4 -3)$ \\ 
		$k=7$ & $\emptyset$ & $\emptyset$ & $\emptyset$ & $\mathcal{B}_{3^4}(3^4 -3)$ \\ 	
		$k=8$ & $\emptyset$ & $\emptyset$ & $\emptyset$ & $\mathcal{B}_{3^4}(3^4 -4)$ \\ 
		$k=9$ & $\emptyset$ & $\emptyset$ & $\emptyset$ & $\mathcal{B}_{3^4}(3^4 -5)$ \\ 	
		$k=10$ & $\emptyset$ & $\emptyset$ & $\emptyset$ & $\mathcal{B}_{3^4}(3^4 -6)$ \\ 
		$k=11$ & $\emptyset$ & $\emptyset$ & $\emptyset$ & $\mathcal{B}_{3^4}(3^4 -7)$ \\ 
		$k=12$ & $\emptyset$ & $\emptyset$ & $\emptyset$ & $\mathcal{B}_{3^4}(3^4 -8)$ \\ 
		$k=13$ & $\emptyset$ & $\emptyset$ & $\emptyset$ & $\mathcal{B}_{3^4}(3^4 -9)$ \\	
		$k=14$ & $\emptyset$ & $\emptyset$ & $\emptyset$ & $\emptyset$	\\[0.1ex] 
	\end{tabular}
\end{table}

\section{Arbitrary natural numbers}

The aim of this section is to complete our investigation by extending Theorem \ref{thm: p-power} to any arbitrary natural number. In order to do this, we first extend Theorem \ref{thm: boxes}.
We recall that $p$ is a fixed odd prime.

\begin{theorem}\label{thm: boxes arb n}
Let $n\in\mathbb{N}$ and let $k\in [0,\alpha_{n}]$. There exists $T_n\6k\in [1,n]$ such that
$\Omega_n\6k=\mathcal{B}_{n}(T_n\6k)$.
Moreover, $T_n\6{k+1}\in\{T_n\6{k}-1, T_n\6{k}\},\ \text{for all}\ k\in [0,\alpha_{n}-1].$
\end{theorem}
\begin{proof}
We proceed by induction on $n\in \mathbb{N}$. If $n=1$, then necessarily $k=0$ and $\Omega_1^0 =\mathcal{B}_1(1)$.
If $n\geq 2$, let $n=\sum_{i=1}^t p^{n_i}$ be the $p$-adic expansion of $n$, with $n_1\geq \dots \geq n_t\geq 0$. By Theorem \ref{thm: boxes}, for every $i\in [1,t]$ and every $d_i \in [0,\alpha_{p^{n_i}}]$, there exists $t_{n_i}^{d_i} \in [\frac{p\6{n_i}}{2}+1, p\6{n_i}]$ such that $\Omega_{p^{n_i}}^{d_i}=\mathcal{B}_{p^{n_i}} \left( t_{n_i}^{d_i} \right) $.
Similarly to the procedure used to prove Theorem \ref{thm: boxes}, we define 
$$\mathcal{J}(k)=\{(j_1,\dots ,j_t)\in \mathcal{C}(k)\ |\ j_i \in [0,\alpha_{p^{n_i}}] \text{   for all } i\in [1,t]\}.$$
Moreover, we set
\[M= \max \Set{\sum_{i=1}^t t_{n_i}^{d_i} | (d_1 , \dots , d_t) \in\mathcal{J}(k)}.\]
We claim that $\Omega_n^k =\mathcal{B}_n(M)$.

Let $(d_1,\ldots, d_t)\in\mathcal{J}(k)$ be such that $M=\sum_{i=1}^t t_{n_i}^{d_i}$. Then 
using Lemma \ref{lem: B star}, Theorem \ref{thm: boxes} and Lemma \ref{lem: Omega star-p-times} we have that
$$\mathcal{B}_n(M)= \mathcal{B}_{p^{n_1}}\left( t_{n_1}^{d_1} \right) \star \cdots \star \mathcal{B}_{p^{n_t}}\left( t_{n_t}^{d_t} \right)=\Omega_{p\6{n_1}}\6{d_1}\star \cdots \star\Omega_{p\6{n_t}}\6{d_t}\subseteq \Omega_n\6k.$$

Suppose now for a contradiction that $\lambda\in \Omega_n^k \smallsetminus \mathcal{B}_n(M)$. Without loss of generality we can assume that $\lambda_1 \geq M+1$.
Let $\phi=\phi_1 \times \cdots \times \phi_t$ be an irreducible constituent of $(\chi^\lambda)_{P_n}$ with $\phi_i(1)=p^{d_i}$ for every $i\in [1,t]$ and $\sum_{i=1}^t d_i =k$. 
Observe that $(d_1,\ldots, d_t)\in\mathcal{J}(k)$.
For every $i\in [1,t]$, let $\mu_i \in \mathcal{P}(p^{n_i})$ be such that $[ (\chi^{\mu_i})_{P_{p^{n_i}}} , \phi_i] \neq 0$ and such that $\chi^{\mu_1}\times \cdots \times \chi^{\mu_t}$ is an irreducible constituent of $(\chi^\lambda)_Y$. Here $Y=\fS_{p^{n_1}}\times\cdots\times\fS_{p^{n_t}}\leq \fS_n$ is chosen so that $P_n\leq Y$. Thus by Theorem \ref{thm: p-power}, $\mu_i \in \Omega_{p^{n_i}}^{d_i}=\mathcal{B}_{p^{n_i}}\left(t_{n_i}^{d_i}\right)$ for every $i\in [1,t]$. Hence,
\[\lambda \in \mathcal{B}_{p^{n_1}} \left(t_{n_1}^{d_1} \right) \star \cdots \star \mathcal{B}_{p^{n_t}} \left(t_{n_t}^{d_t} \right) =\mathcal{B}_n \left( \sum_{i=1}^t t_{n_i}^{d_i} \right) .\]
By Lemma \ref{lem: LR} and our assumptions, we have that
\[M+1 \leq \lambda_1 \leq \sum_{i=1}^t t_{n_i}^{d_i} \leq M ,\]
which is a contradiction.

In summary, for $k\in [0,\alpha_n-1]$ the following holds:
\[ \begin{split} \Omega_n^k = \mathcal{B}_n(M), &\text{ where } M= \max \Set{\sum_{i=1}^t t_{n_i}^{d_i} | (d_1,\dots ,d_t) \in \mathcal{J}(k)}, \text{ and }
\\ \Omega_n^{k+1} =\mathcal{B}_n(T), &\text{ where } T=\max \Set{\sum_{i=1}^t t_{n_i}^{f_i} | (f_1,\dots ,f_t) \in \mathcal{J}(k+1)}.
\end{split} \] 
Let $(d_1,\dots ,d_t)\in \mathcal{J}(k)$ be such that $M= \sum_{i=1}^t t_{n_i}^{d_i}$. 
Since $k\leq \alpha_n-1$, there exists $i\in [1,t]$ such that $d_i\leq \alpha_{p\6{n_i}}-1$. Hence
$(d_1, \dots,d_{i-1}, d_i+1, d_{i+1},\ldots, d_t) \in \mathcal{J}(k+1)$ and $t_{n_i}^{d_i +1} \in \Set{ t_{n_i}^{d_i}-1, t_{n_i}^{d_i} }$,  by Theorem \ref{thm: boxes}. Thus,
$$M-1=-1+\sum_{i=1}^t t_{n_i}^{d_i} \leq t_{n_1}^{d_1} +\cdots +t_{n_{i-1}}^{d_{i-1}} +t_{n_{i}}^{d_{i}+1}+t_{n_{i+1}}^{d_{i+1}}+ \cdots + t_{n_t}^{d_t} \leq T.$$
On the other hand, let $(f_1,\dots ,f_t)\in \mathcal{J}(k+1)$ be such that $T=\sum_{i=1}^t t_{n_i}^{f_i}$. Without loss of generality we can assume that $f_1\geq 1$. Then $(f_1 -1, f_2, \dots , f_t) \in \mathcal{J}(k)$ and by Theorem \ref{thm: boxes}, $t_{n_1}^{f_1} \in \Set{t_{n_1}^{f_1 -1}-1, t_{n_1}^{f_1 -1} }$. Hence
\[T=\sum_{i=1}^t t_{n_i}^{f_i} \leq  t_{n_1}^{f_1 -1}+  t_{n_2}^{f_2} + \cdots +  t_{n_t}^{f_t} \leq M.\]
It follows that $T=M$ or $T=M-1$. This concludes the proof. 	
\end{proof}

Theorem \ref{thm: boxes arb n} shows that for every $n\in\mathbb{N}$ and $k\in [0,\alpha_n]$ there exists an integer, denoted by $T_n\6k$, such that $\Omega_n\6k=\mathcal{B}_n(T_n\6k)$.
In order to prove our main result, i.e. to precisely compute the value $T_n\6k$ for all $n\in\mathbb{N}$ and $k\in [0,\alpha_{n}]$, we start by fixing some notation that will be kept throughout this section. 
We remark that for $n<p\62$ we have that $P_n$ is abelian and that $\Omega_n\60=\mathcal{P}(n)$. For this reason we focus on the case $n\geq p\62$. 

\begin{notation} \label{notation_n} 
Let $n\geq p\62$ be a natural number and let $n=\sum_{i=1}\6{t}p\6{n_i}$ be the $p$-adic expansion of $n$, where $n_1\geq n_2\geq\cdots \geq n_t\geq 0$.
Let $\mathcal{R}:=\{(i, y)\ |\ i\in [1,t],\ \text{and}\ y\in [1,p\6{n_i -2}]\}$. 
We define a total order $\triangleright$ on $\mathcal{R}$ as follows. 
Given $(i, y)$ and $(j, z)$ in $\mathcal{R}$ we say that $(i, y)\triangleright (j, z)$ if and only if one of the following hold: 
\begin{itemize}
\item[(i)] $\ell(n_i,y)>\ell(n_j,z)$, or 

\item[(ii)]  $\ell(n_i,y)=\ell(n_j,z)$ and $i<j$, or 

\item[(iii)] $\ell(n_i,y)=\ell(n_j,z)$ and $i=j$ and $y<z$. 
\end{itemize}

Let $N:=\lfloor \frac{n}{p\62} \rfloor$ and notice that $N=|\mathcal{R}|$. Let $\phi: \mathcal{R}\longrightarrow [1,N]$ be the bijection mapping $(i,y)\mapsto x$ if and only if 
the pair $(i,y)$ is the $x$-th greatest element in the totally ordered set $(\mathcal{R}, \triangleright)$.
We use this bijection to relabel the integers $\ell(n_i, y)$, for all $(i, y)\in \mathcal{R}$. In particular, we let $\ell(x):=\ell(n_i, y)$ if $\phi((i, y))=x$. 
Recalling Definition \ref{def: Ax},
we observe that the definition of $\triangleright$ implies that
$\ell(1)\geq \ell(2)\geq \cdots\geq \ell(N).$

Finally, for any $\alpha\in [1, N]$ we let $F_n(\alpha)=\sum_{a=1}\6\alpha\ell(a)$ and
$A_\alpha=[\{F_n(\alpha-1)+1, F_n(\alpha)\}]$. We observe that $\{A_1, A_2, \ldots, A_{N}\}$ is a partition of $[1,\alpha_n]$ (this follows easily from Lemma \ref{lem: 3x}). 
We refer the reader to Example \ref{ex: final} for an explicit description of these objects in a concrete case.
\end{notation}



\begin{theorem} \label{thm: n}
Let $n\in \mathbb{N}_{\geq p\62}$ and $k\in [1,\alpha_n]$. Let $x\in [1, N]$ be such that $k\in A_x$. Then $$\Omega_n\6k=\mathcal{B}_n(n-x).$$
\end{theorem}
\begin{proof}
As in Notation \ref{notation_n}, let $n=\sum_{i=1}\6{t}p\6{n_i}$ be the $p$-adic expansion of $n$, where $n_1\geq n_2\geq\cdots \geq n_t\geq 0$. 
We proceed by induction on $x$. If $x=1$ then $k\in A_1=[1,\ell(1)]=[1,\ell(n_1,1)]$, because $\phi((n_1,1))=1$. By Theorem \ref{thm: p-power} we know that $\Omega_{p\6{n_1}}\6k=\mathcal{B}_{p\6{n_1}}(p\6{n_1}-1)$. Moreover, $\Omega_{p\6{m}}\60=\mathcal{B}_{p^m}(p\6m)$ for all $m\in\mathbb{N}$ by \cite[Theorem 3.1]{GN}. 
Thus, using first Lemma \ref{lem: Omega star-R-times} and then Lemma \ref{lem: B star}, we deduce that $$\Omega_{n}\6{k}\supseteq\Omega_{p\6{n_1}}\6{k}\star\Omega_{p\6{n_2}}\6{0}\star\cdots\star\Omega_{p\6{n_t}}\6{0}=\mathcal{B}_{p\6{n_1}}(p\6{n_1}-1)\star\mathcal{B}_{p\6{n_2}}(p\6{n_2})\star\cdots\star\mathcal{B}_{p\6{n_t}}(p\6{n_t})=\mathcal{B}_n(n-1).$$
Since $(n)\notin \Omega_n\6k$, we conclude that $\Omega_n\6k=\mathcal{B}_n(n-1)$, as desired. 
Let us now set $x\geq 2$ and assume that the statement holds for any $s\in A_{x-1}=[F_n(x-1)+1,F_n(x)]$. 
From Theorem \ref{thm: boxes arb n} we know that $$\Omega_{n}\6{F_n(x-1)+1}\subseteq \Omega_n\6k\subseteq \Omega_n\6{F_n(x)},$$
hence it is enough to show that:  
$$(1)\ \Omega_n\6{F_n(x-1)+1}=\mathcal{B}_n(n-x),\ \text{and that}\ (2)\ \Omega_n\6{F_n(x)}=\mathcal{B}_n(n-x).$$
Here $F_n(y)=\sum_{j=1}\6y\ell(j)$, exactly as explained in Notation \ref{notation_n}. 

To prove (1), we first notice that $\Omega_n\6{F_n(x-1)}=\mathcal{B}_n(n-(x-1))$ by inductive hypothesis. Hence, Theorem \ref{thm: boxes arb n} implies that $\Omega_n\6{F_n(x-1)+1}=\mathcal{B}_n(T)$, for some $T\in\{n-x, n-(x-1)\}$. Therefore it suffices to prove that $\lambda=(n-(x-1), x-1)\notin \Omega_n\6{F_n(x-1)+1}$. 
Let $\{G_1, G_2, \ldots, G_t\}$ be the partition of $[1,x-1]$ defined by 
$$G_i=\{y\in [1,x-1]\ |\ \phi\6{-1}(y)=(i, z),\ \text{for some}\ z\in [1,p\6{n_i-2}]\},\  \text{for all}\ i\in [1,t].$$
To ease the notation we let $g_i=|G_i|$ for all $i\in [1,t]$, and we remark that $g_1+g_2+\cdots +g_t=x-1$.

Let $Y=\fS_{p\6{n_1}}\times\fS_{p\6{n_2}}\times  \cdots\times \fS_{p\6{n_t}}$ be a Young subgroup of $\fS_n$ containing $P_n$. 
For every $i\in [1,t]$ let $\mu\6i\in\mathcal{P}(p\6{n_i})$ be such that $\mathcal{LR}(\lambda; \mu\61,\ldots, \mu\6t)\neq 0$. Then Lemma \ref{lem: LR} implies that there exist $a_1,a_2,\ldots, a_t\in \mathbb{Z}$ such that $$(\mu\6i)_1=p\6{n_i}-(g_i+a_i)\ \text{for all}\ i\in [1,t],\ \text{and such that}\ \sum_{i=1}\6t a_i\leq 0.$$ 
In particular, using Theorem \ref{thm: p-power} we have that for every $i\in [1,t]$,
$$\mu\6i\in\mathcal{B}_{p\6{n_i}}(p\6{n_i}-(g_i+a_i))\smallsetminus\mathcal{B}_{p\6{n_i}}(p\6{n_i}-(g_i+a_i+1))=\Omega_{p\6{n_i}}\6{f_{n_i}(g_i+a_i)}\smallsetminus\Omega_{p\6{n_i}}\6{f_{n_i}(g_i+a_i)+1}.$$
Recycling the notation used in the proof of Theorem \ref{thm: p-power}, here $f_m(a):=\sum_{j=1}\6a\ell(m, j)$.
It follows that 
\[\partial_{P_{p\6{n_i}}}(\chi\6{\mu\6i})=\sum_{j=1}\6{g_i+a_i}\ell(n_i, j)=\begin{cases}
\sum_{y\in G_i}\ell(y) + \sum_{j=g_i+1}\6{g_i+a_i}\ell(n_i, j) & \mathrm{if}\ a_i\geq 0, \\ 
\\ 
\sum_{y\in G_i}\ell(y) - \sum_{j=g_i+a_i}\6{g_i}\ell(n_i, j) & \mathrm{if}\ a_i< 0. \end{cases}\]

\medskip

Hence, letting $\chi=\chi\6{\mu^1}\times\chi\6{\mu^2}\times\cdots\times\chi\6{\mu^t}$, we have that 
$$\partial_{P_n}(\chi)=\sum_{i=1}\6{t}\sum_{y\in G_i}\ell(y) + E - F,\ \text{where}\ E=\sum_{\substack{i=1 \\ a_i> 0}}\6{t}\sum_{j=g_i+1}\6{g_i+a_i}\ell(n_i, j),\ \text{and}\ F=\sum_{\substack{i=1 \\ a_i< 0}}\6{t}\sum_{j=g_i+a_i}\6{g_i}\ell(n_i, j).$$

We claim that $E-F\leq 0$. To see this, we notice that the definition of the set $G_i$ implies that $\phi((i, y))>x-1$ for all $y\geq g_i+1$. On the other hand, for the same reasons, we have that 
$\phi((j, z))\leq x-1$ for all $z\leq g_j$. Therefore every summand $\ell(n_i, y)$ appearing in $E$ is smaller than or equal to any summand $\ell(n_j, z)$ appearing in $F$. Since $\sum_{i=1}\6ta_i\leq 0$ we have that
$E-F\leq 0$, as desired. 
Using Proposition \ref{prop: partial arbitrary n} we conclude that 
$$\partial_{P_n}(\chi\6\lambda)\leq \sum_{i=1}\6{t}\sum_{y\in G_i}\ell(y)=\sum_{y=1}\6{x-1}\ell(y)=F_n(x-1)<F_n(x-1)+1.$$
Hence $\lambda\notin \Omega_n\6{F_n(x-1)+1}$ and therefore $\Omega_n\6{F_n(x-1)+1}=\mathcal{B}_n(n-x)$ as required.

To prove (2) we observe that the equality (1) shown above implies that $\Omega_n\6{F_n(x)}\subseteq \mathcal{B}_n(n-x)$, by Theorem \ref{thm: boxes arb n}.
To show that the opposite inclusion holds we use an idea that is similar to the one used to prove (1).  
In particular, we let $\{H_1, H_2, \ldots, H_t\}$ be the partition of $[1,x]$ defined by 
$$H_i=\{y\in [1,x]\ |\ \phi\6{-1}(y)=(i, z),\ \text{for some}\ z\in [1,p\6{n_i-2}]\},\  \text{for all}\ i\in [1,t].$$
To ease the notation we let $h_i=|H_i|$ for all $i\in [1,t]$, and we remark that $h_1+h_2+\cdots +h_t=x$.
We also introduce the following notation. For each $i\in [1,t]$, we let $$\Gamma_i:=\sum_{y\in H_i}\ell(y)=\sum_{j=1}\6{h_i}\ell(n_i,j)=f_{n_i}(h_i).$$
We observe that $(\Gamma_1, \Gamma_2, \ldots, \Gamma_t) \in \mathcal{C}(F_n(x))$ and that $\Gamma_i\in [0,\alpha_{p\6{n_i}}]$, for all $i\in [1,t]$. 
We can now use Lemma \ref{lem: Omega star-R-times}, Theorem \ref{thm: p-power} and Lemma \ref{lem: B star} (in this order) to deduce that
$$\Omega_n\6{F_n(x)}\supseteq \Omega_{p\6{n_1}}\6{\Gamma_1}\star\Omega_{p\6{n_2}}\6{\Gamma_2}\star\cdots\star\Omega_{p\6{n_t}}\6{\Gamma_t}=\mathcal{B}_{p\6{n_1}}(p\6{n_1}-h_1)\star\mathcal{B}_{p\6{n_2}}(p\6{n_2}-h_2)\star\cdots\star\mathcal{B}_{p\6{n_t}}(p\6{n_t}-h_t)=\mathcal{B}_n(n-x).$$
We obtain that $\Omega_n\6{F_n(x)}=\mathcal{B}_n(n-x)$, and the proof is concluded. 
\end{proof}

As we have done for the prime power case in Corollary \ref{cor: p^n}, we record some facts to understand better the set $\Omega_n^k$ for every $n\in \mathbb{N}$ and $k\in [0,\alpha_n]$.
Keeping the notation introduced in \ref{notation_n}, we recall that $N=\lfloor\frac{n}{p\62}\rfloor$. 
\begin{corollary} \label{cor: n}
	Let $n\in \mathbb{N}$ and $n=\sum_{i=1}\6{t}p\6{n_i}$ its $p$-adic expansion, where $n_1\geq n_2\geq\cdots \geq n_t\geq 0$. Let $1\leq k<t \leq \alpha_n$. The following hold.
	\begin{itemize}
	\item[(i)] $\mathcal{B}_n(n-N)=\Omega_n^{\alpha_n} \subseteq \Omega_n^t\subseteq \Omega_n^k$.
		\item[(ii)] $\Omega_n^k =\Omega_n^t$ if, and only if, there exists $x \in [1,N]$ such that $k,t \in A_x$.
		\item[(iii)] Given $x\in [1,N]$ we have that $|\Set{k\in [1, \alpha_n] | \Omega_n^k =\mathcal{B}_n(n-x) }| =\ell(x)$.
	\end{itemize}
\end{corollary}
\begin{proof}
	Since $|A_x|=\ell(x)$ for every $x\in [1,N]$, (i), (ii) and (iii) hold by Theorem \ref{thm: n}.
%
\end{proof}

A second consequence of Theorem \ref{thm: n} is the following asymptotic result. This basically says that when $n$ is arbitrarily large, almost all irreducible characters of $\fS_n$ admit constituents of every possible degree on restriction to a Sylow $p$-subgroup. 

\begin{corollary} \label{thm: asintotico}
Let $\Omega_n=\bigcap_k\Omega_n\6k$, where $k$ runs over $[0, \alpha_n]$. Then $$\lim_{n\rightarrow\infty}\frac{|\Omega_n|}{|\mathcal{P}(n)|}=1.$$
\end{corollary}
\begin{proof}
A result of Erd\H{o}s and Lehner \cite[(1.4)]{EL} guarantees that given $f(n)$ a function that diverges as $n$ tends to infinity, then for all but $o(|\mathcal{P}(n)|)$ partitions $\lambda$ of $n$, the quantities $\lambda_1$ and $l(\lambda)$ lie between $\sqrt{n}\cdot( \tfrac{\log n}{d} \pm f(n) )$ where $d$ is a constant. 
By Theorem \ref{thm: n}, we observe that $\Omega_n=\Omega_n\6{\alpha_n}=\mathcal{B}(n-N)$, where $N=\lfloor \frac{n}{p\62} \rfloor$. Since $n-N\geq n/2$, the statement follows.
\end{proof}

\begin{example}\label{ex: final}
	Let $p=3$ and $n=3^3+3^2+3$. Following Notation \ref{notation_n}, we have $n_1=3,\ n_2=2$ and $n_3=1$. Hence $\mathcal{R}=\Set{(1,1),(1,2),(1,3),(2,1)}$, since $[1,3^{n_3 -2}]=\emptyset$. Observe that $|\mathcal{R}|=4=\lfloor \frac{n}{3^2} \rfloor$.
Using Definition \ref{def: Ax}, we can see that $\ell(3,1)=2,\ \ell(3,2)=\ell(3,3)=1$ and $\ell(2,1)=1$.
Hence, the definition of the total order $\triangleright$ on $\mathcal{R}$ implies that $(1,1)\triangleright (1,2)\triangleright (1,3) \triangleright (2,1)$. Thus $\ell(1)=2,\ \ell(2)=\ell(3)=\ell(4)=1$ and 
	\[A_1=\Set{1,2},\ A_2=\Set{3},\ A_3=\Set{4},\ A_4=\Set{5}.\]
	Notice that $\Set{A_1, \dots , A_4}$ is a partition of $[1,\alpha_n]=[1,5]$, as required. 
Moreover by Theorem \ref{thm: n} we have 
	$ \Omega_n^1=\Omega_n^2=\mathcal{B}_n(n-1),\ \Omega_n^3=\mathcal{B}_n(n-2),\ \Omega_n^4=\mathcal{B}_n(n-3),\ \Omega_n^5=\mathcal{B}_n(n-4) $.

Using the notation of Theorem \ref{thm: boxes arb n}, the above computation gives that $T_n^2=n-1$. Following the proof of Theorem \ref{thm: boxes arb n}, we can compute $T_n^2$ in a different way. We have 
	\[\mathcal{J}(2)=\Set{(j_1,j_2,j_3)\in \mathcal{C}(2) | j_1\in [0,4],\ j_2\in[0,1],\ j_3\in \Set{0}}=\Set{(2,0,0),(1,1,0)}.\]
	Hence $M=\max \Set{t_3^2+t_2^0+t_1^0, t_3^1+t_2^1+t_1^0}=\max \Set{n-1, n-2} =n-1$. Thus $T_n^2=M=n-1$, as expected.
Notice that $n_3=1$ does not contribute at all to the computations. In fact in $\mathcal{R}$ there are no elements of the form $(3,y)$, $y\in \mathbb{N}$. Furthermore, by looking at the third column of Table \ref{table: final}, we can see that $T_n^k=T_{n-3}^k -3$ for every $k\in \mathbb{N}$.
A second example of this fact can be found by observing that the first two columns of Table \ref{table: final} are equal. 
\end{example}

\begin{table}[ht]
	\caption{Let $p=3$. According to Theorem \ref{thm: n}, the structure of $\Omega_n^k$ is recorded in the entry corresponding to row $k$ and column $n$.} \label{table: final}
	\centering
	\begin{tabular}{ c || c | c | c | c | c } 
		$\Omega_{n}^k$ &$n=3+3^3$ &$n=2\cdot 3 +3^3$ &$n=3^2+3^3$ &$n=3^3+3^3$ &$n=3^3+3^4$  \\[0.1ex]
		\hline 
		\ & \ & \ & \ & \ & \   \\
		$k=0$ & $\mathcal{B}_n(n)$ & $\mathcal{B}_n(n)$ & $\mathcal{B}_n(n)$ & $\mathcal{B}_n(n)$ & $\mathcal{B}_n(n)$ \\ 
		$k=1$ & $\mathcal{B}_n(n -1)$ & $\mathcal{B}_n(n -1)$ & $\mathcal{B}_n(n -1)$ & $\mathcal{B}_n(n -1)$ & $\mathcal{B}_n(n -1)$  \\ 
		$k=2$ & $\mathcal{B}_n(n -1)$ & $\mathcal{B}_n(n -1)$ & $\mathcal{B}_n(n -1)$ & $\mathcal{B}_n(n -1)$ & $\mathcal{B}_n(n -1)$ \\
		$k=3$ & $\mathcal{B}_n(n -2)$ & $\mathcal{B}_n(n -2)$ & $\mathcal{B}_n(n -2)$ & $\mathcal{B}_n(n -2)$ & $\mathcal{B}_n(n -1)$  \\
		$k=4$ & $\mathcal{B}_n(n -3)$ & $\mathcal{B}_n(n -3)$ &$\mathcal{B}_n(n -3)$ & $\mathcal{B}_n(n -2)$ & $\mathcal{B}_n(n -2)$  \\
		$k=5$ & $\emptyset$ & $\emptyset$ & $\mathcal{B}_n(n -4)$ & $\mathcal{B}_n(n -3)$ & $\mathcal{B}_n(n -2)$  \\
		$k=6$ & $\emptyset$ & $\emptyset$ & $\emptyset$ & $\mathcal{B}_n(n -4)$ & $\mathcal{B}_n(n -3)$ \\
		$k=7$ & $\emptyset$ & $\emptyset$ & $\emptyset$ & $\mathcal{B}_n(n -5)$ & $\mathcal{B}_n(n -3)$  \\
		$k=8$ & $\emptyset$ & $\emptyset$ & $\emptyset$ & $\mathcal{B}_n(n -6)$ & $\mathcal{B}_n(n -4)$  \\
		$k=9$ & $\emptyset$ & $\emptyset$ & $\emptyset$ & $\emptyset$ & $\mathcal{B}_n(n -4)$  \\
		$k=10$ & $\emptyset$ & $\emptyset$ & $\emptyset$ & $\emptyset$ & $\mathcal{B}_n(n -5)$ \\
		$k=11$ & $\emptyset$ & $\emptyset$ & $\emptyset$ & $\emptyset$ & $\mathcal{B}_n(n -6)$ \\
		$k=12$ & $\emptyset$ & $\emptyset$ & $\emptyset$ & $\emptyset$ & $\mathcal{B}_n(n -7)$ \\
		$k=13$ & $\emptyset$ & $\emptyset$ & $\emptyset$  & $\emptyset$ & $\mathcal{B}_n(n -8)$ \\
		$k=14$ & $\emptyset$ & $\emptyset$ & $\emptyset$ & $\emptyset$ & $\mathcal{B}_n(n -9)$ \\
		$k=15$ & $\emptyset$ & $\emptyset$ & $\emptyset$ & $\emptyset$ & $\mathcal{B}_n(n -10)$ \\
		$k=16$ & $\emptyset$ & $\emptyset$ & $\emptyset$ & $\emptyset$ & $\mathcal{B}_n(n -11)$ \\
		$k=17$ & $\emptyset$ & $\emptyset$ & $\emptyset$ & $\emptyset$ & $\mathcal{B}_n(n -12)$ \\
		$k=18$ & $\emptyset$ & $\emptyset$ & $\emptyset$ & $\emptyset$ & $\emptyset$  \\[0.1ex] 
	\end{tabular}
\end{table}

\end{document}